\documentclass[10pt]{amsart}
\usepackage{latexsym}
\usepackage{amsmath,color}
\usepackage{amssymb}
\usepackage[all]{xy}

\newtheorem{thm}{Theorem}[section]

\newtheorem{lem}[thm]{Lemma}
\theoremstyle{definition}
\newtheorem{qu}[thm]{Question}
\newtheorem{defn}[thm]{Definition}

\newtheorem{rem}[thm]{Remark}
\newtheorem{ex}[thm]{\bf Example}

\newtheorem{prob}[thm]{Problem}
\newcommand{\Q}{{\mathbb Q}}

\newcommand{\N}{{\mathbb N}}

\newcommand{\C}{{\mathbb C}}
\newcommand{\e}{\varepsilon}
\newcommand{\mapright}[1]{%
 \smash{\mathop{%
  \hbox to 1cm{\rightarrowfill}}\limits_{#1}}}
\newcommand{\maprightd}[2]{%
 \smash{\mathop{%
  \hbox to 1.2cm{\rightarrowfill}}\limits^{#1}\limits_{#2}}}
\newcommand{\mapleft}[1]{%
 \smash{\mathop{%
  \hbox to 1cm{\leftarrowfill}}\limits_{#1}}}
\newcommand{\mapleftu}[1]{%
 \smash{\mathop{%
  \hbox to 1cm{\leftarrowfill}}\limits^{#1}}}
\newcommand{\mapleftud}[2]{%
 \smash{\mathop{%
  \hbox to 1.2cm{\leftarrowfill}}\limits^{#1}\limits_{#2}}}

\begin{document}

\title{ Certain  maps preserving self-homotopy equivalences}
\author{Jin-ho Lee \ and \ Toshihiro Yamaguchi}

\footnote[0]{2010 {\it Mathematics Subject Classification}.  55P62, 55P10\\
\,\,\,\,{\it Key words and phrases}.\ self homotopy equivalence, $\mathcal{E}$-map, co-$\mathcal{E}$-map,
 rational homotopy, Sullivan  (minimal) model, rational $\mathcal{E}$-map, rational co-$\mathcal{E}$-map,
rationally  $\mathcal{E}$-equivalent}
\date{}

\address{Department  of  Mathematics, Korea University, Seoul, Korea}
\email{sabforev@korea.ac.kr}

\address{Faculty of Education, Kochi University,  Kochi, Japan}
\email{tyamag@kochi-u.ac.jp}

\maketitle
\begin{abstract}Let $\mathcal{E}(X)$ be
the group  of homotopy classes of
self  homotopy
equivalences for   a connected CW complex $X$.
We observe two classes of maps $\mathcal{E}$-maps and co-$\mathcal{E}$-maps.
They are defined as the maps $X\to Y$ that induce the homomorphisms
$\mathcal{E}(X)\to \mathcal{E}( Y)$
and  $\mathcal{E}(Y)\to \mathcal{E}(X)$, respectively.
We  give some rationalized examples  related to  spheres, Lie groups and homogeneous spaces  by using Sullivan models.
Furthermore,  we introduce an  $\mathcal{E}$-equivalence relation between rationalized  spaces $X_{\Q}$ and $Y_{\Q}$
as a geometric realization of an isomorphism  $\mathcal{E}(X_{\Q})\cong \mathcal{E}(Y_{\Q})$.
\end{abstract}

\section{Introduction}

Needless to say,
the based  homotopy set $[X,Y]$ of based continuous maps from a based space $X$ to a based space $Y$
is a most interesting  object in homotopy theory.
In the following, all maps are based and we do not distinguish a homotopy class and the representative in a homotopy set.
Let $X$ be a connected CW complex with  base point $\ast$ and
let
$$\mathcal{E}(X)=\{ [f]\in [X,X]\ |\ f:X\overset{\simeq  }{\to}  X \}$$
be the group of homotopy classes
of self-homotopy equivalences for $X$ with the operation given by the composition of homotopy classes.
This group is important and has been closely studied as part of homotopy theory
(for example, see \cite{AC}, \cite{MR}, \cite{P}, \cite{R}, \cite{b1}, \cite{b2},  \cite{BS}).

It is clear that   $\mathcal{E}(X) \cong \mathcal{E}(Y)$ as a group
if $X\simeq Y$.
One of the difficulties  of its computation or evaluation may be based on the fact that
$\mathcal{E}(\ \ )$ is not functorial,
i.e.,
there is no  suitable induced map between   $\mathcal{E}(X)$ and $\mathcal{E}(Y)$
for the map $f:X\to Y$ in general.
However,
recall that, for example,   the injection $i_X:X\to X\times Y$ and the projection $p_Y:X\times Y\to Y$
 induce the natural  monomorphisms $\mathcal{E}(X)\to \mathcal{E}(X\times Y)$
and $\mathcal{E}(Y)\to \mathcal{E}(X\times Y)$, respectively.

\begin{defn}\label{def}
{\it We say a map $f:X\to Y$ is   an  {\bf   $\mathcal{E}$-map} if
there is a homomorphism
$\phi_f :\mathcal{E}(X)\to \mathcal{E}(Y)$
such that $$\xymatrix{
X\ar[r]^g\ar[d]_{f}&X\ar[d]^{f}\\
Y\ar[r]^{\phi_f(g)}&Y
}$$
homotopically commutes
for any element $g$ of $\mathcal{E}(X)$.
We say the map $f:X\to Y$ is   a {\bf  co-$\mathcal{E}$-map} if
there is a homomorphism
$\psi_f :\mathcal{E}(Y)\to \mathcal{E}(X)$
such that  $$\xymatrix{
X\ar[r]^{\psi_f(g)}\ar[d]_{f}&X\ar[d]^{f}\\
Y\ar[r]^g&Y
}$$
homotopically commutes
for any element $g$ of $\mathcal{E}(Y)$.}
\end{defn}

Especially  we consider the
rationalized version of $\mathcal{E}$-maps
and co-$\mathcal{E}$-maps by using Sullivan models \cite{FHT}, \cite{GM}, \cite{Su}.
Let $X_{\Q}$ be the rationalization of a nilpotent space $X$ \cite{HMR}.

\begin{defn}\label{rdef}
{\it A map $f:X \to Y$ between nilpotent spaces
is a {\bf  rational $\mathcal{E}$-map} if the rationalization
$f_\Q :X_\Q \to Y_\Q$ is an $\mathcal{E}$-map.
Similarly a map $f:X \to Y$ between nilpotent spaces
is a {\bf rational co-$\mathcal{E}$-map} if the rationalization
$f_\Q:X_\Q \to Y_\Q$ is a co-$\mathcal{E}$-map.}

\end{defn}


\begin{qu}
When is a map a (rational) $\mathcal{E}$-map or a (rational) co-$\mathcal{E}$-map ?
\end{qu}

\begin{thm}\label{hg}Let $G$ be a compact connected Lie group and
$H$ be a connected closed sub-Lie group of  $G$.

(1)
The inclusion $j:H\to G$ is a rational $\mathcal{E}$-map if  $\pi_*(j)\otimes \Q$ is injective.

(2)
For the   homogeneous space $G/H$,
 the projection map $f:G \to G/H$ is a rational co-$\mathcal{E}$-map.
\end{thm}


The assumption $\mathcal{E}(X) \cong \mathcal{E}(Y)$ does not, in general, imply $X$ and $Y$ are homotopy equivalent spaces.
Finally, we consider the question: {\it When is an isomorphism $\mathcal{E}(X_{\Q}) \cong \mathcal{E}(Y_{\Q})$
realized as a composition of rational $\mathcal{E}$-maps and   rational co-$\mathcal{E}$-maps
between $X_{\Q}$ and $Y_{\Q}$?}

\begin{defn}\label{equ}
We say that nilpotent spaces $X$ and $Y$ are {\bf rationally $\mathcal{E}$-equivalent} (denote as
$X_{\Q}\underset{\mathcal{E}}{\sim}Y_{\Q}$) if
there is a chain of  {\it spherically  injective} $\mathcal{E}$-maps and  {\it  spherically  injective} co-$\mathcal{E}$-maps  $$X_{\Q}\overset{f_1}{\to} Z_1\overset{f_2}{\gets} \cdots \overset{f_n}{\gets}  Z_n\overset{f_{n+1}}{\to}  Y_{\Q}$$ ($Z_i$ are rational spaces)
such that an isomorphism $\mathcal{E}(X_{\Q}) \cong \mathcal{E}(Y_{\Q})$
is given by a composition of $n+1$-isomorphisms   $\{ \phi_{f_i}\}_i$ and   $\{ \psi_{f_i}\}_i$, i.e.,
$\phi_{f_{n+1}}\circ \psi_{f_n}\circ \cdots \circ \psi_{f_2}\circ \phi_{f_1}: \mathcal{E}(X_{\Q}) \overset{\cong}{\to} \mathcal{E}(Y_{\Q})$ or $\psi_{f_{1}}\circ \phi_{f_2}\circ \cdots \circ \phi_{f_n}\circ \psi_{f_{n+1}}: \mathcal{E}(Y_{\Q}) \overset{\cong}{\to} \mathcal{E}(X_{\Q})$.
\end{defn}

\begin{rem}
In this paper, we say that a map $f:X\to Y$ is   {\it  spherically  injective} when $f_{\sharp}(u)\neq 0\in \pi_*(Y)$
if $hur_X(u)\neq 0$ for $u\in \pi_*(X)$.
Here $hur_X:\pi_*(X)\to H_*(X)$ is the Hurewicz homomorphism for a space $X$.  Thus we have 
$$ \mbox{(weakly) homotopy equivalent}\Rightarrow \mbox{spherically injective} \Rightarrow \mbox{homotopy non-trivial}$$
If we admit the homotopy trivial maps as $f_i$,  any isomorphism $\mathcal{E}(X_{\Q}) \cong  \mathcal{E}(Y_{\Q})$ induces 
$X_{\Q} \underset{\mathcal{E}}{\sim} Y_{\Q}$ by the constant map $*:X_{\Q}\to Y_{\Q}$.
\end{rem}

\begin{thm}\label{sp}
(1) If $X$ and $Y$ are rationally homotopy equivalent, i.e., $X_{\Q}\simeq Y_{\Q}$,  then $X_{\Q}\underset{\mathcal{E}}{\sim}Y_{\Q}$.\\
(2) For any $n$, $S^2_{\Q}\underset{\mathcal{E}}{\sim}\C P^n_{\Q}$ and $S^4_{\Q}\underset{\mathcal{E}}{\sim}{\mathbb H} P^n_{\Q}$.\\
(3) When $m$ is even and $n$ is odd, $(S^m\vee S^n)_{\Q}\underset{\mathcal{E}}{\sim}(S^m\times  S^n)_{\Q}$
if and only if  $m\neq  n+1$.\\
(4) For   odd-integers $1<m\leq n$, $(S^m\times  S^n)_{\Q}\underset{\mathcal{E}}{\sim} E_{\Q}\underset{\mathcal{E}}{\sim} E'_{\Q}$
for  non-trivial fibrations $S^{m+n-1}\to E\overset{p}{\to}  S^m\times  S^n$ and $S^{2m+n-2}\to E'\overset{p'}{\to}  E$.\\
(5) There are integers  $m\neq n$ such  that   $S^m_{\Q}\underset{\mathcal{E}}{\sim}S^n_{\Q}$.
For example,  $S^{53}_{\Q}\underset{\mathcal{E}}{\sim}S^{67}_{\Q}$.

\end{thm}

\begin{rem}
The proof of Theorem \ref{sp} (5) requires a rigid rational   space $X$  of \cite{AL}, which induces 
$\mathcal{E}(X)=\{ id_X\}$.
For the total space  $Z_1$ of a fibration $$X\to Z_1\to S^m_{\Q}\times S^n_{\Q},$$
it is given by the sequence 
$S^{m}_{\Q}\overset{f_1}{\to} Z_1\overset{f_2}{\gets}S^{n}_{\Q}$ 
of an $\mathcal{E}$-map $f_1$ with $\phi_{f_1}:\mathcal{E}(S^{m}_{\Q})  \cong \mathcal{E}(Z_1)$ and a co-$\mathcal{E}$-map $f_2$ with $\psi_{f_2}:\mathcal{E}(Z_1)  \cong \mathcal{E}(S^{n}_{\Q})$.
Because of the demand that $\mathcal{E}(Z_1) \cong \Q^*$  ($\Q^*:=\Q -0$), we need  suitable  restrictions about the pair 
 $(m,n)$
and  see that $(53,67)$ satisfies them in the proof.
Of course, it depends on the model structure of $X$.
So we may require more various types of rigid models for the proof of (5) in many cases of  $(m,n)$.
On the other hand, the authors cannot find an example that  $S^m_{\Q}\underset{\mathcal{E}}{\not\sim}S^n_{\Q}$
for some $(m,n)$.
\end{rem}

\begin{prob}
If $\mathcal{E}(X) \cong  \mathcal{E}(Y)$ for rational spaces $X$ and $Y$,
does it hold that $X \underset{\mathcal{E}}{\sim} Y$ ?
\end{prob}



\begin{rem}
For rational spaces $X$, $Y$ and $Z$,
even if $ Y\underset{\mathcal{E}}{\sim} Z$,
it may not hold that $X\times Y\underset{\mathcal{E}}{\sim}X\times Z$.
For example, when $X=S^5$, $Y=S^2$ and $Z=\C P^2$,
$\mathcal{E}((X\times Y)_{\Q})\cong \Q^*\times \Q^*$
but $\mathcal{E}((X\times Z)_{\Q})$
 is isomorphic to the subgroup of lower  triangular matrixes
of $GL(2,\Q )$.
\end{rem}

In \S 2, we demonstrate the basic properties
and provide examples in ordinary homotopy theory of $\mathcal{E}$-maps and co-$\mathcal{E}$-maps.
In \S 3, we give some computations in rational homotopy theory using Sullivan minimal models.



\section{Some properties}

Recall that $[X,\ \ ]$ is the covariant functor from the category of spaces to the category of sets,
where  for a map $f:Y\to Z$,
the map $f_*(g):[X,Y]\to [X,Z]$ is given by
 $f_*(g)=f\circ g$.
On the other hand, $[\ \ ,Z]$ is the cotravariant functor.
For the map $f:X\to Y$,
 the map $f^*(g):[Y,Z]\to [X,Z]$ is given by $f^*(g)=g\circ f$.
The following lemma holds from   $\phi_f(g)\circ f=f\circ g$ and $f\circ \psi_f(g)=g\circ f$.

\begin{lem}
A map $f:X\to Y$
is an $\mathcal{E}$-map $($or a co-$\mathcal{E}$-map$)$ if and only if
there is a group homomorphism $\phi_f:\mathcal{E}(X)\to \mathcal{E}(Y)$ $($or
 $\psi_f:\mathcal{E}(Y)\to \mathcal{E}(X))$
where the following diagrams
 $$\xymatrix{
[X,X]\ar[r]^{f_*}&[X,Y]& [Y,Y]\ar[l]_{f^*}\\
\mathcal{E}(X)\ar[u]^{\cup}\ar[rr]^{\phi_f}&&\mathcal{E}(Y)\ar[u]_{\cup}
}\ \ \ \ \
\xymatrix{
[X,X]\ar[r]^{f_*}&[X,Y]& [Y,Y]\ar[l]_{f^*}\\
\mathcal{E}(X)\ar[u]^{\cup}&&\mathcal{E}(Y)\ar[u]_{\cup}\ar[ll]_{\psi_f}
}$$
are commutative.
\end{lem}

Of course, the maps $\phi_f$ and $\psi_f$ may not be uniquely determined
for a map $f$.
\begin{lem}\label{prop}
(1) If  maps $f:X\to Y$ and  $g:Y\to Z$ are   $\mathcal{E}$-maps,
then $g\circ f:X\to Z$ is an  $\mathcal{E}$-map.

(2) If $f$ and $g$ are co-$\mathcal{E}$-maps,
then $g\circ f$ is a co-$\mathcal{E}$-map.

(3) The  constant map is both an  $\mathcal{E}$-map and
a  co-$\mathcal{E}$-map.

(4) A homotopy equivalence map is both an  $\mathcal{E}$-map and
a  co-$\mathcal{E}$-map.
\end{lem}

\begin{proof}
(1) $\phi_{g\circ f}(h):=\phi_g\circ \phi_f(h)$ for $h\in \mathcal{E}(X)$.\\
(2) $\psi_{g\circ f}(h):=\psi_f\circ \psi_g(h)$ for $h\in \mathcal{E}(Z)$.\\
(3) It is sufficient to put $\phi_f=\psi_f=*$, i.e.,
$\phi_f(g)=id_Y$ and  $\psi_f(g)=id_X$ for any $g$.\\
(4)  It is sufficient to put  $\phi_{ f}(h):=f\circ h\circ f^{-1}$  for $h\in \mathcal{E}(X)$
and
 $\psi_{f}(h):=f^{-1}\circ h\circ f$ for $h\in \mathcal{E}(Y)$, where $f^{-1}$ is the homotopy inverse of $f$.
\end{proof}

\begin{defn}\cite[Chapter 3]{H}(\cite{K2})
Let $\alpha:X \to Y$ and $\beta:Z \to W$ be maps.
$\Pi(\alpha,\beta)$ is the set of all homotopy classes of pairs $[f_1,f_2]$ such that

$$\xymatrix{
X \ar[r]^{f_1} \ar[d]_\alpha & Z \ar[d]^\beta\\%
Y \ar[r]^{f_2} & W}$$
\noindent
is homotopy commutative.
Here a homotopy of $(f_1,f_2)$ is just a pair of homotopies $(f_{1t},f_{2t})$  such that
$\beta f_{1t}=f_{2t}\alpha$.
If $[f_1,f_2]$ has a two sided inverse in $\Pi(\alpha,\beta)$,
we call $[f_1,f_2]$ a homotopy equivalence.
If $\alpha=\beta$, we call $[f_1,f_2]$  a self-homotopy equivalence and denote the set of all self-homotopy equivalences by $\mathcal{E}(\alpha)$.

\end{defn}

\begin{lem}
For a map $f:X \to Y$,\\
(1) $f$ is an $\mathcal{E}$-map if and only if $h:\mathcal{E}(f) \to \mathcal{E}(X)$
given by $h[g_1,g_2]=[g_1]$  is an epimorphism with a section.\\
(2)  $f$ is a co-$\mathcal{E}$-map if and only if $h':\mathcal{E}(f) \to \mathcal{E}(Y)$
given by  $h'[g_1,g_2]=[g_2]$ is an epimorphism with a section.
\end{lem}

\begin{proof}
(1) Suppose that $f$ is an $\mathcal{E}$-map.
Then we have a map $\phi_f: \mathcal{E}(X) \to \mathcal{E}(Y)$ such that $\phi_f(g) \circ f \simeq f \circ g$ for any $g \in \mathcal{E}(X)$.
Thus we have $[g,\phi_f(g)]\in \mathcal{E}(f)$ and $h[g,\phi_f(g)]=[g]$ and $h$ is epimorphic.
Next  suppose that $h$ is an epimorphism.
For any $[g] \in \mathcal{E}(X)$,
we have $[g',g''] \in \mathcal{E}(f)$ such that $h[g',g'']=[g]$.
So $g$ is homotopic to $g'$.
Since $[g',g''] \in \mathcal{E}(f)$,
$g'$ and $g''$ are homotopy equivalences and $g '' \circ f \simeq f \circ g'$.
Thus we can define a map $\phi_f:\mathcal{E}(X) \to \mathcal{E}(Y)$ by $\phi_f(g)= \pi \circ s[g]$
where $\pi:\mathcal{E}(f) \to \mathcal{E}(Y)$ is  the natural projection
and $s$ is  the section of the assumption. Hence, $f$ is an $\mathcal{E}$-map.

(2) Suppose that $f$ is a co-$\mathcal{E}$-map.
Then we have a map $\psi_f: \mathcal{E}(Y) \to \mathcal{E}(X)$ such that $g \circ f \simeq f \circ \psi_f(g)$
for any $g \in \mathcal{E}(Y)$. So we have $[\psi_f(g),g]\in \mathcal{E}(f)$ and $h'[\psi_f(g),g]=[g]$.
Thus $h'$ is epimorphic.
Next  suppose that $h'$ is an epimorphism.
For any $[g] \in \mathcal{E}(Y)$, we have $[g',g''] \in \mathcal{E}(f)$ such that $h[g',g'']=[g]$
and thus $g$ is homotopic to $g''$.
Since $[g',g''] \in \mathcal{E}(f)$, $g'$ and $g''$ are homotopy equivalences and $g '' \circ f \simeq f \circ g'$.
Then we can define a map $\psi_f:\mathcal{E}(Y) \to \mathcal{E}(X)$ by $\psi_f(g)= h \circ s'[g]$
for the section $s'$.
Hence, $f$ is a co-$\mathcal{E}$-map.
\end{proof}

\begin{thm}Let $\eta:S^3 \to S^2$ and $\nu:S^7 \to S^4$ be the Hopf fibrations with fibre $S^1$ and $S^3$, respectively.
Let   $\epsilon_3 :S^{11}\to S^3$ be   the generator of $\pi_{11}(S^3)\cong\mathbb{Z}_2$(\cite{T}).
Then\\
(1) $\eta$ is a co-$\mathcal{E}$-map, but not an $\mathcal{E}$-map,\\
(2) $\nu$ is neither an $\mathcal{E}$-map nor a co-$\mathcal{E}$-map and\\
(3) $\epsilon_3$ is both an $\mathcal{E}$-map and a co-$\mathcal{E}$-map.
\end{thm}

\begin{proof}
(1) From \cite[Example 4.2 (i)]{KO}, we have $\Pi(\eta,\eta)=\{(k^2 \iota_3, k \iota_2)\,|\,k \in \mathbb{Z}\}$ as a set.
Therefore, we have a homotopy commutative diagram
$$\xymatrix{
S^3 \ar[r]^{k^2 \iota_3 } \ar[d]_\eta & S^3 \ar[d]^\eta\\%
S^2 \ar[r]^{k\iota_2} & S^2 }$$
It is well known that $\mathcal{E}(S^n)=\{\iota_n,-\iota_n\} \cong \mathbb{Z}_2$.
Since $(\iota_3,-\iota_2),(\iota_3,\iota_2) \in \Pi(\eta,\eta)$, $\eta$ is a co-$\mathcal{E}$-map.
However, there is no map $f:S^2 \to S^2$ such that $(-\iota_3,f) \in \Pi(\eta,\eta)$.
Thus $\eta$ is not an $\mathcal{E}$-map.

(2) From  \cite[Example 4.2 (ii)]{KO},
we have $\Pi(\nu,\nu)=\{(k^2 \iota_7, k \iota_4)\,|\,k(k-1)\equiv 0(mod\,\,8)\}$ as a set.
Therefore, we have a homotopy commutative diagram
$$\xymatrix{
S^7 \ar[r]^{k^2 \iota_7 } \ar[d]_\nu & S^7 \ar[d]^\nu\\%
S^4 \ar[r]^{k\iota_4} & S^4 }$$
Since there are no maps $f:S^7 \to S^7$ and $g:S^4 \to S^4$ such that
$(f,-\iota_4),(-\iota_7,g) \in \Pi(\nu,\nu)$,
$\nu$ is neither an $\mathcal{E}$-map nor a co-$\mathcal{E}$-map.

(3) From  \cite[Example 4.2 (iv)]{KO},
we have
$\Pi(\epsilon_3,\epsilon_3)=\{((d+2s) \iota_{11}, d \iota_3)\,|\, d,s \in \mathbb{Z}\} \cong \mathbb{Z} \times \mathbb{Z}$ as a group.
Therefore we have a homotopy commutative diagram
$$\xymatrix{
S^{11} \ar[rr]^{(d+2s) \iota_{11} } \ar[d]_{\epsilon_3} && S^{11} \ar[d]^{\epsilon_3} \\%
S^3 \ar[rr]^{d\iota_3} && S^3 }$$
Since $(\iota_{11},\iota_3),(-\iota_{11},-\iota_3) \in \Pi(\e_3,\e_3)$,
$\e_3$ is both an $\mathcal{E}$-map and a co-$\mathcal{E}$-map.
\end{proof}

\begin{ex}
(1) Let $e:X \to \Omega \Sigma X$ be the adjoint of $id_{\Sigma X}$ from the one-to-one correspondence $[X, \Omega \Sigma X] \cong [\Sigma X, \Sigma X]$.
We know that $e(x)(t)=\langle x,t\rangle$.
Let $f$ be a self homotopy equivalence on $X$, that is, $ f \in \mathcal{E}(X)$ and let $f'$ be a homotopy inverse of $f$.
It is clear that the map $\Sigma f: \Sigma X \to \Sigma X$, $\Sigma f\langle x,t \rangle =\langle f(x),t\rangle$, is a homotopy equivalence with homotopy inverse $\Sigma f'$.
Then we define a map $\widetilde{f}:\Omega\Sigma X \to  \Omega\Sigma X$ by $\widetilde{f}(\alpha)(t)=\Sigma f (\alpha (t))$.
Define another map $\widetilde{f'}: \Omega \Sigma X \to  \Omega \Sigma X$ by $\widetilde{f'}(\alpha)(t)=\Sigma f' (\alpha (t))$.
Clearly we have $\widetilde{f} \circ \widetilde{f'} \simeq id$ and $\widetilde{f'} \circ \widetilde{f} \simeq id$.
Moreover we have
$e(f(x))(t)=\langle f(x),t\rangle$
and
$\tilde{f}( e(x))(t) =\Sigma f (e(x)(t))=\Sigma f \langle x,t \rangle =\langle f(x),t \rangle$.
Therefore we have a commutative diagram
$$\xymatrix{
X \ar[rr]^{f} \ar[d]_e && X \ar[d]^e \\%
\Omega \Sigma X \ar[rr]^{\widetilde{f}} && \Omega \Sigma X}$$
Thus $e:X \to \Omega \Sigma X$ is an $\mathcal{E}$-map.\\

(2) Let $\pi:\Sigma \Omega Y \to Y$ be the adjoint of $id_{\Omega Y}$ from the one-to-one correspondence $[\Sigma \Omega Y, Y] \cong [\Omega Y, \Omega Y]$.
We know that $\pi\langle \alpha,t \rangle =\alpha(t)$.
Let $g$ be a self homotopy equivalence on $Y$, that is $ g \in \mathcal{E}(Y)$ and let $g'$ be a homotopy inverse of $g$.
Then we define a map $\widetilde{g}:\Sigma \Omega Y \to \Sigma \Omega Y$ by $\widetilde{g}\langle \alpha,t \rangle=
\langle g \circ \alpha, t\rangle$ and
$\widetilde{g'}:\Sigma \Omega Y \to \Sigma \Omega Y$ by $\widetilde{g'} \langle \alpha,t
\rangle =\langle g' \circ \alpha, t\rangle$.
Clearly we have $\widetilde{g} \circ \widetilde{g'} \simeq id$ and $\widetilde{g'} \circ \widetilde{g} \simeq id$.
Moreover we have
$(\pi \circ \widetilde{g})\langle \alpha,t\rangle =\pi\langle g \circ \alpha,t \rangle =(g \circ \alpha)(t)$
and
$(g \circ \pi) \langle \alpha,t \rangle =g(\alpha(t))$.
Therefore we have a commutative diagram
$$\xymatrix{
\Sigma \Omega Y \ar[rr]^{\widetilde{g}} \ar[d]_\pi && \Sigma \Omega Y \ar[d]^\pi \\%
Y \ar[rr]^g && Y}$$
Therefore $\pi:\Sigma \Omega Y \to Y$ is a co-$\mathcal{E}$-map.
\end{ex}

\begin{ex}
There is a natural homomorphism
$\mathcal{E}(X_{(n)})\to \mathcal{E}(X_{(n-1)})$
obtained by restricting the map  to a lower
 Postnikov section  \cite[p.27]{AC}.
 Thus  the principal $K(\pi_n(X),n)$-fibration
$X_{(n)}\to X_{(n-1)}$
is an $\mathcal{E}$-map.
The restriction map $X\to X_{(n-1)}$ is also an $\mathcal{E}$-map.
On the other hand, for the $n$-skeleton $X^{(n)}$,
the inclusions
$X^{(n)}\to X^{(n+1)}$ and $X^{(n)}\to X$
are  both  co-$\mathcal{E}$-maps.
\end{ex}

\begin{rem}Recall that a space $X$ is said to be
(homotopically) {\it rigid}
when  $\mathcal{E}(X)=\{ id_X\}$ \cite{CV} (\cite{AL}).
 If $X$ is rigid, then every map $f:X\to Y$  is an ${\mathcal E}$-map by $\phi_f(id_X)=id_Y$.
 If $Y$ is rigid, then  every map $f:X\to Y$  is a co-${\mathcal E}$-map by $\psi_f(id_Y)=id_X$.
Also we can construct infinitely many examples of ${\mathcal E}$-maps
and co-${\mathcal E}$-maps  by using  the functor of \cite[Remark 2.8]{CV}.
\end{rem}

\section{Computations in  Sullivan models}

We assume that $X$ is a nilpotent CW complex.
Let
$M(X)=(\Lambda {V},d)$
be the  Sullivan minimal model of $X$ \cite{Su}.
  It is a free $\Q$-commutative differential graded algebra (DGA)
 with a $\Q$-graded vector space $V=\bigoplus_{i\geq 1}V^i$
 where $\dim V^i<\infty$ and a decomposable differential; i.e., $d(V^i) \subset (\Lambda^+{V} \cdot \Lambda^+{V})^{i+1}$ and $d \circ d=0$.
 Here  $\Lambda^+{V}$ is
 the ideal of $\Lambda{V}$ generated by elements of positive degree.
The degree of a homogeneous element $x$ of a graded algebra is denoted as $|{x}|$.
Then  $xy=(-1)^{|{x}||{y}|}yx$ and $d(xy)=d(x)y+(-1)^{|{x}|}xd(y)$.
Note that  $M(X)$ determines the rational homotopy type of $X$.
In particular,  $H^*(\Lambda {V},d)\cong H^*(X;\Q )$
and $V^i\cong Hom(\pi_i(X),\Q)$.
Refer to  \cite{FHT} for details.

Let ${\rm Aut}M$ be the group of DGA-automorphisms of a DGA $M$.
For a nilpotent space $X$ and a (not necessarily  minimal) model
$M(X)$, there is a group isomorphism
$$\mathcal{E}(X_{\Q})\cong \mathcal{E}( M(X)),$$
where $\mathcal{E}( M(X))={\rm Aut}M(X)/\sim$ is the group of  self-DGA-homotopy equivalence classes
of $M(X)$ \cite{Su}.
Now  recall about ``DGA-homotopy'' in \cite{GM}:
In general, two maps $f:M(Y)\to M(X)$ and $g:M(Y)\to M(X)$ are DGA-homotopic (denote as $f\sim g$) if there is a DGA-map
$H:M(Y)\to M(X)\otimes \Lambda (t,dt)$ such that
$H\mid_{t=0,dt=0}=f$ and $H\mid_{t=1,dt=0}=g$.
Here $|t|=0$ and $|dt|=1$ with $d(t)=dt$, $d(dt)=0$.

The group $\mathcal{E}( M(X))$ does not depend on choosing a model of $X$.
For example,  the minimal model $M=M(S^{2n+1})=(\Lambda w,0)$ and a  non-minimal model  $M'=(\Lambda (y,w,v),D)$
with $|y|=2$, $|w|=2n+1$,  $|v|=1$, $Dy=0$,  $Dw=y^{n+1}$ and $Dv=y$ are both  models of $S^{2n+1}$.
Obviously we have $\mathcal{E}( M)\cong \Q^*$ by $w\to aw$ for $a\in \Q^*$.
On the other hand, in $\mathcal{E}(M')$
we can set
$H(y)=cyt+cvdt$, $H(v)=cvt$ and $H(w)=aw+by^nv+b'y^nvt^{n+1}$ with $a,b,b',c\in \Q^*$.
Then $a+b=0$ and $b'=c^{n+1}$ from $D\circ H=H\circ D$.
Thus two maps $f\in {\rm Aut}M'$ given by $f(y)=f(v)=0$, $f(w)=a(w-y^nv)$
and $g\in {\rm Aut}M'$ given by $g(y)=cy$, $g(v)=cv$, $g(w)=a (w-y^nv)+c^{n+1}y^nv$
are DGA-homotopic.
Hence we have   $\mathcal{E}(M')=\mathcal{E}(M)\cong \Q^*$
as in  Example \ref{Hopf}(1) below.

\begin{rem}
From the universality of the localization \cite{HMR},
the rationalization map $l:X\to X_{\Q}$ is an  $\mathcal{E}$-map,
but it is not  a co-$\mathcal{E}$-map in general.
For example,
when $X=S^3$,
 the elements $f$ of  $\mathcal{E}( M(X))=\mathcal{E}(\Lambda (x),0)$
with $f(x)=ax$ for  $a\neq \pm 1\in \Q^*= \Q -0$
can not be realized as a homotopy  equivalence of $X$.
\end{rem}

 The model   of a map $f: X\to  Y$ between nilpotent spaces
 is given by
a relative model:
$$M(Y)=(\Lambda W,d_Y)\overset{i}\to
 (\Lambda W\otimes \Lambda V,D)\overset{q}\to  (\Lambda V,\overline{D})$$
 with  $D|_{\Lambda W}=d_Y$ and the minimal model
 $ (\Lambda V,\overline{D})$ of the homotopy fiber of $f$.
It is well known that there is a quasi-isomorphism
$M(X)\to (\Lambda W\otimes \Lambda V,D)$ \cite{FHT}.
Then   Definition \ref{rdef}  is   translated  to

\begin{defn}\label{model}
Let  $f:X\to Y$ be a map  between nilpotent spaces.

(1) The map $f$ is    a { rational $\mathcal{E}$-map} if and only if
there is a homomorphism
$\phi_f :\mathcal{E}(\Lambda W\otimes \Lambda V,D)\to \mathcal{E}(\Lambda W,d_Y)$
such that
$$\xymatrix{(\Lambda  W\otimes \Lambda V,D)
\ar[r]^g& (\Lambda  W\otimes \Lambda V,D)\\
 (\Lambda  W,d_Y)\ar[u]^i\ar[r]^{\phi_f(g)}& (\Lambda  W,d_Y)\ar[u]_i
 }$$
is DGA-homotopy commutative for any element $g$ of  $\mathcal{E}(\Lambda W\otimes \Lambda V,D)$.

(2) The map $f$ is    a { rational co-$\mathcal{E}$-map} if and only if
there is a homomorphism
$\psi_f :\mathcal{E}(\Lambda W,d_Y)\to \mathcal{E}(\Lambda W\otimes \Lambda V,D)$
such that
$$\xymatrix{(\Lambda  W\otimes \Lambda V,D)
\ar[r]^{\psi_f(g)}& (\Lambda  W\otimes \Lambda V,D)\\
 (\Lambda  W,d_Y)\ar[u]^i\ar[r]^g& (\Lambda  W,d_Y)\ar[u]_i
 }$$
is DGA-homotopy commutative for any element $g$ of $\mathcal{E}(\Lambda W,d_Y)$.
 \end{defn}

\begin{ex}\label{Hopf}
(1) For the Hopf fibration $S^1\to S^{2n+1}\overset{f}\to  \C P^n$,
the relative model is given by
 $$(\Lambda (y,w),d_Y)\to
 (\Lambda (y,w,v),D)\to  (\Lambda (v),0)$$
with $|y|=2$, $|w|=2n+1$,  $|v|=1$, $d_Yw=y^{n+1}$ and $Dv=y$.
We can identify $\mathcal{E}(\C P^n_{\Q})$ as $\Q^*$
by $g(y)=ay$ and $g(w)=a^{n+1}w$ for $g\in \mathcal{E}(\C P^n_{\Q})$
and $a\in\Q^*$.
Also we have $\mathcal{E}(S^{2n+1}_{\Q})= \mathcal{E}(\Lambda (y,w,v),D)=\mathcal{E}(\Lambda w,0)\cong \Q^*$.
 Then there is a homomorphism
$$\psi_{f} :\Q^*\cong \mathcal{E}(\C P^n_{\Q})\to \mathcal{E}(S^{2n+1}_{\Q})\cong \Q^*$$
which is given by $\psi_{f} (a)=a^{n+1}$ for $a\in \Q^*$.
Thus $f$ is a rational co-$\mathcal{E}$-map,
but it is not a rational $\mathcal{E}$-map.

(2) Let $X$ be the pullback of the sphere bundle of the tangent bundle of $S^{m+n}$ by the canonical
degree $1$ map $S^m\times S^n\to S^{m+n}$ for odd integers $m$ and $n$.
Then it is the total space of a  fibration
$S^{m+n-1}\to X\overset{f}\to  S^m\times S^n$
whose model is $$(\Lambda (w_1,w_2),0)\to
 (\Lambda (w_1,w_2,u),D)\to  (\Lambda (u),0)$$
with $|w_1|=m$, $|w_2|=n$, $|u|=m+n-1$ and $Du=w_1w_2$
 is both a rational $\mathcal{E}$-map
and a rational co-$\mathcal{E}$-map.

(3) The  fibration $S^m\times  S^{m+n-1} \to  X\overset{f}\to  S^n$ ($m\neq n$ are odd) 
 whose model is $$(\Lambda (w),0)\to
 (\Lambda (w,v,u),D)\to  (\Lambda (v,u),0)$$
where  $|w|=n$, $|v|=m$, $|u|=m+n-1$ and $Du=wv$ with $m$, $n$ odd
is both a rational $\mathcal{E}$-map and a rational co-$\mathcal{E}$-map.

(4) For the fibration $\C P^{n-1}\to \C P^{2n-1}\overset{f}\to  S^{2n}$
given by $$(\Lambda (y,w),d_Y)\to
 (\Lambda (y,w,x,v),D)\to  (\Lambda (x,v),\overline{D})$$
with $|y|=2n$, $|w|=4n-1$, $|x|=2$, $|v|=2n-1$, $d_Yy=0$, $d_Yw=y^2$, 
$Dx=\overline{D}x=0$, $Dv=y-x^n$ and $\overline{D}v=x^n$,
the map $f$ is a rational $\mathcal{E}$-map given by $\phi_f(a)=a^n$ for $a\in \Q^*$
but not  a rational co-$\mathcal{E}$-map.
\end{ex}

\begin{ex}\label{ex2}
For an $n$-dimensional manifold $X$,
the collapsing map of lower cells $f:X\to  S^{n}$
is an $\mathcal{E}$-map.
Indeed, from the commutative diagram between cofibrations
$$\xymatrix{
X^{(n-1)}\ar[d]\ar[rr]^{g\mid_{X^{(n-1)}}}&& X^{(n-1)}\ar[d]\\
X\ar[d]_f\ar[rr]^g && X\ar[d]^f\\
S^n\ar@{.>}[rr]^{\overline{g}} && S^n,
}$$
we have $\phi_f(g)=\overline{g}$,
but it is  not a (rational) co-$\mathcal{E}$-map in general.
For example, the collapsing map of lower cells $f: X = \mathbb{C}P^n \to S^{2n}=Y$ induces a DGA-map
$$f^*:M(Y)=(\Lambda (y,w),d_Y) \to (\Lambda (x,v),d_X)=M(X)$$
with $d_Y w =y^2$, $d_X v =x^{n+1}$, $f^\ast(y) = x^n$ and $f^\ast (w) =x^{n-1} v$.
The map $f$ is  a rational $\mathcal{E}$-map by $\phi_f(a)=a^n$ for $a\in\Q^*$
but not  a rational co-$\mathcal{E}$-map.
Indeed, for $g^*(y)=ay$ with $a\not\in (\Q^*)^{\times m}:=\Q^*\cdot \Q^* \cdots \Q^*$ (m-times),
we cannot define $\psi_{f^*}(g^*)$.
\end{ex}

\begin{ex}
Let $\Omega Y=map ((S^1,*),(Y,*))$ be the base point preserving the loop space of
a simply connected space $Y$ and
$LY=map (S^1,Y)$,  the free loop space of $Y$.
We consider the evaluation map
$f:LY\to  Y$ with $f(\sigma )=\sigma (*)$.
It is a co-$\mathcal{E}$-map
by
$\psi_f(g)(h)=g\circ h$ for
$g\in \mathcal{E}(Y)$.
 What is the (rational) homotopical condition of $Y$ that allows $f$ to be a (rational) $\mathcal{E}$-map?
According to \cite{VS},
the relative model of the free loop fibration
$\Omega Y\to LY\overset{f}\to  Y$:
$$ M(Y)=(\Lambda V,d)\to (\Lambda V\otimes \Lambda \overline{V},D)\to (\Lambda \overline{V},0)$$
 with $M(LY)=(\Lambda V\otimes \Lambda \overline{V},D)$
is defined as follow:
The graded vector space $\overline{V}$ satisfies
$\overline{V}^n\cong V^{n-1}$ for $n>0$
and denote by $s:V\to \overline{V}$ ($s(v):=\overline{v}$) this isomorphism of degree $-1$.
There is a unique extension  of $s$ into a derivation of algebra $s:\Lambda V\otimes \Lambda \overline{V}
\to \Lambda V\otimes \Lambda \overline{V}$
 such that $s(\overline{V})=0$.
The differential $D$ is given by
$D(v)=d(v)$ for $v\in V$ and $D({\overline{v}})=-s\circ d (v)$
for ${\overline{v}}\in {\overline{V}}$.

If every  DGA-isomorphism $g$ of  $(\Lambda V\otimes \Lambda \overline{V},D)$
 satisfies  $g|_{ \Lambda V}\in \mathcal{E}(\Lambda V,d)$,
 then   $f$ ($M(f)$)  is a rational $\mathcal{E}$-map
by $\phi_f(g)=g|_{ \Lambda V}$.\\
(1) When $Y=S^n$, we observe that the map $f$ is a rational $\mathcal{E}$-map.
If $n$ is even,  $M(S^n)=(\Lambda (x,y),d)$ with $|x|=n$, $|y|=2n+1$, $dx=0$ and $dy=x^2$.
For example,  when $n=2$, note that there is no  DGA-map $g(x)=x+\overline{y}$.\\
(2) When  $Y=S^m\times S^n$ for odd integers $m< n$, the map $f$ is a rational $\mathcal{E}$-map
 if and only if
$m-1$ is not a  divisor of $n-1$.
Indeed,    let $M(S^m\times S^n)=(\Lambda (x,y),0)$.
When $n-1=a(m-1)$
for an integer $a>1$, there is a DGA-isomorphism  $g:(\Lambda (x,y,\overline{x},\overline{y}),0)\to (\Lambda (x,y,\overline{x},\overline{y}),0)$
with $g(x)=x$, $g( \overline{x})=\overline{x}$, $g( \overline{y})=\overline{y}$
 and $g (y)=y+\overline{x}^{a-1}x$.
 Then  $f$ cannot be a rational $\mathcal{E}$-map.
 When $n-1\neq a(m-1)$ for any $a$, a self-map $g$ is given by $g(x)=x$ and $g(y)=y$
from the degree reason.

\end{ex}
\vspace{0.2cm}

\noindent{\it Proof of Theorem \ref{hg}(1)}.
Note that $\pi_*(j)_{\Q}$ is injective if and only if the model of $j:H\to G$ is given as the projection
$M(G)\cong (\Lambda (v_1,\cdots ,v_k,u_1,\cdots , u_l),0)\to (\Lambda (v_1,\cdots ,v_k),0)\cong M(H)$
after a suitable basis change.
Then we can define as $\phi_j(g)=g\otimes 1_{\Lambda (u_1,\cdots ,u_l)}$
for any $g\in \mathcal{E}(\Lambda (v_1,\cdots ,v_k),0)$.
\hfill\qed\\

For the  $n$-dimensional  unitary group $U(n)$,
 $M(U(n))=M(S^1\times \cdots \times S^{2n-1})=(\Lambda (v_1,\cdots , v_n),0)$
with $|v_i|=2i-1$.
For the  n-dimensional special unitary group $SU(n)$,
 $M(SU(n))=(\Lambda (v_1,\cdots , v_{n-1}),0)$
with $|v_i|=2i+1$.
For the  n-dimensional symplectic  group $Sp(n)$,
 $M(Sp(n))=(\Lambda (v_1,\cdots , v_n),0)$
with $|v_i|=4i-1$.

\begin{ex}
In general,  for a connected closed sub-Lie group $H$ of a compact connected Lie group $G$,
the inclusion $j:H\to G$, is not a rational $\mathcal{E}$-map.
For example, the blockwise inclusion $j:SU(3)\times SU(3)\to SU(6)$ is not.
Indeed, $M(SU(3)\times SU(3))=(\Lambda (u_1,w_1,u_2,w_2),0)$
with $|u_1|=|w_1|=3$ $|u_2|=|w_2|=5$ and  $M(SU(6))=(\Lambda (v_1,v_2,v_3,v_4,v_5),0)$
with $|v_i|=2i+1$.
$M(j)(v_i)=u_i+w_i$ for $i=1,2$.
Then we cannot define $\phi_j(g)$ for $g\in \mathcal{E}(\Lambda (u_1,w_1,u_2,w_2),0)$
when  $g(u_i)=u_i$ $g(w_i)=-w_i$ for example.
\end{ex}

\begin{lem}\label{odd2}
Let  $X=S^{a_1}\times \cdots \times S^{a_m}\times Y$ and
$Y= S^{b_1}\times  \cdots \times S^{b_n}$
for odd-integers $a_1\leq \cdots \leq a_m\leq b_1\leq \cdots \leq b_n$.
 Then  the second  factor projection  map $f:X\to Y$
is a rational {$\mathcal{E}$-map} if and only if there are  no subsets
$\{ i_1, \cdots , i_k\}$    of $\{ 1,\cdots , m\}$
and $\{j_1,\cdots ,j_k\}$
of $\{ 1,\cdots , n\}$  with
 $b_k=a_{i_1}+\cdots +a_{i_k}+b_{j_1}+\cdots +b_{j_k}$ for $k=1,..,n$.
\end{lem}

\begin{proof}
Put $M(X)=(\Lambda (x_1,..,x_m, y_1,..,y_n),0)$ and
$M(Y)=(\Lambda (y_1,..,y_n),0)$
with $|x_i|=a_i$ and $|y_i|=b_i$.
If  $b_k=a_{i_1}+\cdots +b_{j_k}$,
there is a map $g\in \mathcal{E}(M(X))$ such that
$$g(x_i)=x_i\ \ (i\leq m),\ \ \ g(y_i)=y_i\ \ (i\neq k),\ \ \ g(y_k)=y_k+x_{i_1}\cdots x_{i_k}y_{j_1}\cdots y_{j_k}$$
and $M(f)(y_i)=y_i$ for all $i$.
Then  we can not have a DGA-homotopy commutative diagram
$$\xymatrix{(\Lambda  (x_1,..,x_m,y_1,..,y_n),0)
\ar[r]^{g}& (\Lambda   (x_1,..,x_m,y_1,..,y_n),0)\\
 (\Lambda   (y_1,..,y_n),0)\ar[u]^{M(f)}\ar[r]^{\phi_f (g)}& (\Lambda   (y_1,..,y_n),0).\ar[u]_{M(f)}
 }$$
 If  $b_k\neq a_{i_1}+\cdots +b_{j_k}$ for any $k$ and   index set,
we can put $$\phi_f(g)=g\mid_{\Lambda (y_1,..,y_n)}$$
in the diagram  for any map $g\in \mathcal{E}(M(X))$.
\end{proof}

\begin{thm}(1)
When $2<m<n$,
the natural  projection
$p_{n,m}:U(n)\to U(n)/U(m)$ is a rational $\mathcal{E}$-map if and only if $n<5$.

(2) When $2<m<n$,
the natural  projection
$p_{n,m}:SU(n)\to SU(n)/SU(m)$ is a rational $\mathcal{E}$-map if and only if $n<8$.
\end{thm}

\begin{lem}\label{odd}
Let  $X=S^{a_1}\times \cdots \times S^{a_m}$ and
$Y=X\times S^{b_1}\times  \cdots \times S^{b_n}$
for odd-integers $a_1\leq \cdots \leq a_m\leq b_1\leq \cdots \leq b_n$.
 Then  the first factor inclusion  map $f:X\to Y$
is a rational {co-$\mathcal{E}$-map} if and only if there is no subset $\{ i_1, \cdots , i_k\}$
of $\{ 1,\cdots , m\}$
with  $b_k=a_{i_1}+\cdots +a_{i_k}$ for $k=1,..,n$.
\end{lem}

\begin{proof}
Let  $M(X)=(\Lambda (x_1,..,x_m),0)$ and
$M(Y)=(\Lambda (x_1,..,x_m,
y_1,..,y_n),0)$
with $|x_i|=a_i$ and $|y_i|=b_i$.
If  $b_k=a_{i_1}+\cdots +a_{i_k}$,
there is a map $g\in \mathcal{E}(M(Y))$ such that
$$g(x_i)=x_i\ \ (i\leq m),\ \ \ g(y_i)=y_i\ \ (i\neq k),\ \ \ g(y_k)=y_k+x_{i_1}\cdots x_{i_k}$$
and $M(f)(x_i)=x_i$  and $M(f)(y_i)=0$ for all $i$.
Then  we cannot have a DGA-homotopy commutative diagram
$$\xymatrix{(\Lambda  (x_1,..,x_m),0)
\ar[r]^{\psi_f(g)}& (\Lambda   (x_1,..,x_m),0)\\
 (\Lambda   (x_1,..,x_m,y_1,..,y_n),0)\ar[u]^{M(f)}\ar[r]^g& (\Lambda   (x_1,..,x_m,y_1,..,y_n),0).\ar[u]_{M(f)}
 }$$
 If  $b_k\neq a_{i_1}+\cdots +a_{i_k}$ for any $k$ and   $\{ i_1, \cdots , i_k\}$,
we can put $$\psi_f(g)=g\mid_{\Lambda (x_1,..,x_m)}$$
in the diagram  for any map $g\in \mathcal{E}(M(Y))$.
\end{proof}

From Lemma \ref{odd}, we have the following.

\begin{thm}
(1) When $2<m<n$,
 the natural inclusion map $i_{m,n}:U(m)\to U(n)$
is   a rational co-$\mathcal{E}$-map if and only if $n<5$.

(2)
When $2<m<n$,
 the natural inclusion map $i_{m,n}:SU(m)\to SU(n)$
is   a rational co-$\mathcal{E}$-map if and only if $n<8$.

(3)
When $m\leq 4$,
 the natural inclusion map $i_{m,n}:Sp(m)\to Sp(n)$
is   a rational co-$\mathcal{E}$-map for any $m\leq n$.
When $4<m<n$,
 the natural inclusion map $i_{m,n}:Sp(m)\to Sp(n)$
is   a rational co-$\mathcal{E}$-map if and only if $n<14$.
\end{thm}

\begin{proof}
(3) For $S=\{ 3,7,11,15,19,23,27,31,35,39,43,47,51,55, \cdots \}$,
there are no integers $a,b,c,d\in S$ with $a<b<c<d$ satisfying the equation $a+b+c=d$
since
$$(4i-1)+(4j-1)+(4k-1)=4(i+j+k)-3\neq 4l-1 $$
for any $i,j,k,l\in \N$.
On the other hand,
$3+7+11+15+(19+4i)=55+4i=|v_{14+i}|$
for $i\geq 0$.
\end{proof}

For a connected closed sub-Lie group $H$ of a compact connected Lie group $G$
with inclusion $j:H\to G$,
there is the induced map $Bj:BH\to BG$ between the classifying spaces.
It induces a map
$Bj^* :M(BG)=(\Lambda V_{BG},0)=(\Q [x_1,\cdots,x_k],0)\to (\Lambda V_{BH},0)=M(BH)$
between the models.
Here $|x_i|$ are even and ${\rm rank} G=k$.
Let $V_G^n=V_{BG}^{n+1}$ by corresponding $y_i$ to $x_i$ with $|y_i|=|x_i|-1$.

\begin{lem}(\cite[Proposition 15.16]{FHT})\label{homog}
The (non-minimal) model of $G/H$ is given as
$(\Lambda V_{BH}\otimes \Lambda V_G,d)$
where $dx_i=0$ and $dy_i=Bj^*(x_i)$ for $i=1,..,k$.
\end{lem}

\noindent{\it Proof of Theorem \ref{hg}(2)}.
For $f:G\to G/H$,  $M(f)$ is given by the projection
$(\Lambda V_{BH}\otimes \Lambda V_G,d)\to (\Lambda V_G,0)$
sending elements of $\Lambda V_{BH}$ to zero
from Lemma \ref{homog}.
Thus we can define $\psi_f(g)$ for any $g\in \mathcal{E}(\Lambda V_{BH}\otimes \Lambda V_G,d)$
by $\psi_f(g)=\overline{g}$
because $g(x_i)\in\Q [x_1,..,x_k]$.
\hfill\qed\\

\begin{ex}
Let $X$ be a $G$-space for a Lie group $G$.
When is the orbit map $f:X\to X/G$  a rational  co-$\mathcal{E}$-map ?
Let $X=S^2\times S^3$, where $M(S^2\times S^3)=(\Lambda (x,y,z),d)$
with $dx=dz=0$ and $dy=x^2$ of $|x|=2$, $|y|=|z|=3$.
There are  free $S^1$-actions on $X$ where
 $M(X/S^1)=M(ES^1\times_{S^1}X)=(\Lambda (t,x,y,z),D)$
 for $M(BS^1)=(\Q[t],0)$ with $|t|=2$ \cite{AP}, \cite{Hal}.
If  the Borel space of a $S^1$-action has the model with $Dx=Dt=0$, $Dy=x^2$ and $Dz=t^2$
(it is given by a free action on $S^3$),
$f$ is not a rational  co-$\mathcal{E}$-map.
Indeed,
we
can not define $\psi_f(g)$ for the DGA-map
$g$ with $g(x)=t$,  $g(t)=x$,  $g(y)=z$ and  $g(z)=y$.
But if   a $S^1$-action has the model with  $Dy=x^2+at^2$ and $Dz=xt$ for $a\not\in \Q^*/(\Q^*)^2$,
the orbit map
$f$ is  a rational  co-$\mathcal{E}$-map.
\end{ex}


\begin{rem}\label{even}Even if a map $f$ is an  $\mathcal{E}$-map,
it may not be  a rational   $\mathcal{E}$-map.
Recall an example of  Arkowitz and   Lupton \cite{AL}:
Let
 $Y$ be a rational (non-universal)  space
such that  $M(Y)=(\Lambda (x_1,x_2,y_1,y_2,y_3,z),d)$
with $|x_1|=10$,   $|x_2|=12$,  $|y_1|=41$,  $|y_2|=43$,  $|y_3|=45$,  $|z|=119$,
$$dx_1 =dx_2= 0, \ dy_1 = x^3_1x_2, \ dy_2 = x^2_1x^2_2,\ dy_3 = x_1x^3_2\ \ \ \mbox{ and}$$
$$ dz = x_2(y_1x_2-x_1y_2)(y_2x_2-x_1y_3) + x^{12}_1 + x^{10}_2.$$
Then   $\mathcal{E}(Y)= \{ g_1,g_2\}(\cong \{  1,-1\})$ 
where $g_1=id_Y$ and   $g_2$ is given by $$g_2(x_1)= x_1, \ g_2(x_2)= -x_2, \ g_2(y_1)= -y_1,$$
$$g_2(y_2)= y_2, \ g_2(y_3)=- y_3, \ g_2(z)= z$$ \cite[Example 5.2]{AL}.
Consider the 12-dimensional homotopy generator
$f: S^{12}\to Y$ corresponding to $x_2$.
It is  an  $\mathcal{E}$-map
by the homotopy commutative diagram:
$$\xymatrix{
S^{12}\ar[rr]^{\pm 1}\ar[d]_{f}&&S^{12}\ar[d]^{f}\\
Y\ar[rr]^{\phi_f(\pm 1)}&&Y
}$$
where  $\phi_f:\mathcal{E}(S^{12})=\{ \pm 1\}\cong \mathcal{E}(Y)$ by $\phi_f(1)=g_1$ and  $\phi_f(-1)=g_2$.
But it is not a rational   $\mathcal{E}$-map.
Because there is no map $\phi_f:{\mathcal{E}(M(S^{12}))}= \mathcal{E}({\Lambda}(u,v),d)\to \mathcal{E}(M(Y)) $  when $a\neq\pm 1\in \Q^*$,
i.e., there is no homotopy commutative diagram:
$$\xymatrix{
(\Lambda (u,v),d)\ar[rr]^{\times a}&&(\Lambda (u,v),d)\\
M(Y)\ar[u]^{M(f)}\ar[rr]^{\phi_f(\times a)}&&M(Y)\ar[u]_{M(f)}
}$$
where
$M(f)(x_2)=u$, $M(f)(z)=u^8v$  and $M(f)$ sends the others to zero.
Here ${M(S^{12})}=({\Lambda}(u,v),d)$ with  $|u|=12$, $|v|=23$,
$du=0$,  $dv=u^2$
and $\times a(u)=au$, $\times a(v)=a^2v$.
\end{rem}





\noindent
{\it Proof of Theorem \ref{sp}} \ 
(1) It  is obvious from the definition.

(2) As (non-graded) DGAs,   $M(\C P^n)\cong M({\mathbb H} P^n)\cong (\Lambda (x,y),d)$ where $dx=0$ and $dy=x^{n+1}$. 
Therefore the inclusions $S^2\to \C P^n$ and $S^4\to {\mathbb H} P^n$
 induce ${\mathcal{E}}(S^2_{\Q})\cong {\mathcal{E}}(\C P^n_{\Q})(\cong \Q^*:=\Q -0)$ and 
${\mathcal{E}}(S^4_{\Q})\cong {\mathcal{E}}({\mathbb H} P^n_{\Q})(\cong \Q^*)$, respectively.

(3) Suppose that $m$ is even and $n$ is odd.
Then $M(S^m\times  S^n)=(\Lambda (x,y,z),d)$ where $dx=dy=0$ and $dz=x^2$  with $|x|=m$, $|y|=n$ and $|z|=2m-1$.

If $m-1=n$, we have  $|z|=|xy|$.
Then  any element of 
${\mathcal{E}}(\Lambda (x,y,z),d)$ is given as 
$$x\to ax, \ \  y\to by,\ \ z\to a^2z+cxy$$ for some  $a,b\in \Q^*$ and $c\in \Q$.
The same is not true for ${\mathcal{E}}((S^m\vee S^n)_{\Q})={\mathcal{E}}((\Lambda (x,y,z,\cdots ),d)$
since $[xy]=0$ in $H^*(S^m\vee S^n;\Q )=\Q [x]\otimes \Lambda (y)/(x^2,xy)$, i.e, $c=0$.

If 
$2m-1=n$,  we have $|z|=|y|$.
Then  any element of 
${\mathcal{E}}(\Lambda (x,y,z),d)$ is given as 
$$x\to ax, \ \  y\to by,\ \ z\to a^2z+cy$$ for  some $a,b\in \Q^*$ and $c\in \Q$.
The same is true for ${\mathcal{E}}((S^m\vee S^n)_{\Q})={\mathcal{E}}((\Lambda (x,y,z,\cdots ),d)$,
which  is also checked by using its Quillen model \cite{FHT},\cite{BS}.

In the other case, $c=0$ for $S^m\vee S^n$ and $S^m\times  S^n$.

Thus we have the following  table:
{\small
\begin{center}
\begin{tabular}{|c|c||c|}
\hline
${\mathcal{E}}((S^m\vee S^n)_{\Q})$&${\mathcal{E}}((S^m\times  S^n)_{\Q})$&$m$:even,  $n$:odd \\
\hline
$\Q^* \times \Q^*$&$\Q^*\times \Q^*\times \Q$&$m-1=n$ \\
\hline
$\Q^* \times \Q^*\times \Q$&$\Q^*\times \Q^*\times \Q$& $2m-1=n$\\
\hline
$\Q^*\times \Q^*$&$\Q^*\times \Q^*$& other \\
\hline
\end{tabular}
\end{center}
}
Hence the inclusion $f:S^m\vee S^n\to S^m\times  S^n$ induces
 $\phi_f: {\mathcal{E}}((S^m\vee S^n)_{\Q})\cong {\mathcal{E}}((S^m\times  S^n)_{\Q})$ if and only if $m\neq n+1$. 

(4) Since the fibrations are non-trivial,  
the models of the total spaces are uniquely determined as 
$M(E)\cong (\Lambda  (x,y, v_1),d)$ with $|x|=m$, $|y|=n$, $dx=dy=0$ and $dv_1=xy$
and $M(E')\cong (\Lambda  (x,y, v_1,v_2),d)$ with $dx=dy=0$, $dv_1=xy$, $dv_2=xv_1$.
Then $\psi_{p'_{\Q}}\circ \psi_{p_{\Q}}: {\mathcal{E}}(\Lambda  (x,y),0)\cong  {\mathcal{E}}(\Lambda  (x,y, v_1),d)\cong  {\mathcal{E}}(\Lambda  (x,y, v_1,v_2),d)$
from degree reasons.


(5) Recall the rigid model of Arkowitz and  Lupton \cite{AL}:
Let $$M = (\Lambda (x_1, x_2, y_1, y_2, y_3, z),d)$$ with 
given by $|x_1| = 8$, $|x_2| = 10$, $|y_1| = 33$, $|y_2| = 35$, $|y_3| = 37$ and
$|z| = 119$, 
$$dx_1 =dx_2= 0, \  dy_1 = x^3_1x_2, \ dy_2 = x^2_1x^2_2,\ 
dy_3 = x_1x^3_2 \ \ \ \mbox{ and} $$
 $$dz = \alpha :=x^4_1(y_1x_2-x_1y_2)(y_2x_2-x_1y_3) + x^{15}_1 + x^{12}_2.$$
(Note that the degrees of elements are determined by the differential  $d$.) 
Then   ${\mathcal{E}}(M)=\{ id_M\}$ \cite[Example 5.1]{AL}.
Now define a Hirsch extension \cite{GM} by $M$ as 
$$(\Lambda (v,w),0)\to  (\Lambda (v,w)\otimes M ,d')=:M'$$ where
$|v|$ and $|w|$ are suitable  odd integers with $|v|\neq |w|$, $|v|+|w|=120$ and 
$d'v=d'w=0$, $d'z=\alpha +vw$ ($d'=d$ for the other elements). 
For example,
put $|v|=53$ and $|w|=67$.
Then any DGA-automorphism $h$  of $M'$ is DGA-homotopic to a map with no unipotent part. Indeed, 
let $$h(z)=z+f_1v+f_2w,\ \ h(v)=av,\ \ h(w)=bw+f_3v$$
where  $a,b\in \Q^*$ with $ab=1$ and $f_1,f_2,f_3\in (x_1,x_2)$, which is the ideal generated by $x_1,x_2$.
Then, since $|f_1|=66$, $|f_2|=52$ and $|f_3|=14$,  we have $$f_1=k_1x_1^2x_2^5+l_1x_1^7x_2,\ \ f_2=k_2x_1^4x_2^2, \ \ f_3=0$$
for $k_1,l_1,k_2\in \Q$.
Thus $f_1$ and $f_2$ are $d'$-exact cocycles.
Therefore $h$ is DGA-homotopic to the map with $f_1=f_2(=f_3)=0$ \cite{AL},\cite{GM}.
Hence any element  $h\in \mathcal{E}(M') $
 is determined by  $h(v)=av$ and $h(w)=bw$ for $a,b\in \Q^*$ such that
$ab=1$
(since $h\mid_M\sim id_M$).
Thus we obtain  $${\mathcal{E}}(M')= \{ (a,b)\in \Q^*\times \Q^* \mid ab=1\}\cong \Q^*.$$
Therefore the DGA-surjections  $(\Lambda v,0)\overset{f^*}\gets M'\overset{g^*}\to (\Lambda w,0)$ 
(spherically injective maps $f: S^{|v|}_{\Q}\to ||M'||$ and $g: S^{|w|}_{\Q}\to ||M'||$, which are their geometric realizations) 
with $f^*(M)=g^*(M)=0$, $f^*(v)=v$, $g^*(w)=w$ and $f^*(w)=g^*(v)=0$ 
induce 
$$(\Q^*\cong )\ {\mathcal{E}}(S^{|v|}_{\Q})={\mathcal{E}}(\Lambda v,0)\underset{\phi_f}{\cong} {\mathcal{E}}(M')\underset{\psi_g}{\cong} {\mathcal{E}}(\Lambda w,0)={\mathcal{E}}(S^{|w|}_{\Q})$$
with $\psi_g \phi_f(a)=a^{-1}$ for $a\in\Q^*$.
Thus we have  $S^{|v|}_{\Q}\underset{\mathcal{E}}{\sim}S^{|w|}_{\Q}$.
\hfill\qed\\


{\noindent}
{\bf Acknowledgement}

The authors are grateful to Antonio Viruel  for his  interest  and suggesting  examples in \cite{CV}
and Jim Stasheff for his kind support. 

\end{document}